\documentclass[a4paper,11pt]{article}
\usepackage{amsmath,textcomp,latexsym,graphics}
\usepackage{amsfonts}
\usepackage{amssymb,amsthm}
\usepackage{amssymb}

\DeclareGraphicsExtensions{.pdf,.png,.jpg}

\def\dis{\displaystyle}
\def\ep{\varepsilon}

\def\R{\mathbb{R}}
\def\T{\mathbb{T}}

\def\A{{\mathcal A}}

\def\te{\theta}

\newenvironment{proofof}[1]{\bigskip\noindent{\it Proof of~#1.}\rm}{\hfill $\Box$}
\newtheorem{theorem}{Theorem}[section]
\newtheorem{proposition}[theorem]{Proposition}
\newtheorem{lemma}[theorem]{Lemma}
\newtheorem{remark}{Remark}
\title{On the maximization of a class of functionals on convex regions,
and the characterization of the farthest convex set}
\author{Evans M. Harrell II\\
School of Mathematics\\
Georgia Institute of Technology\\
Atlanta GA, 30332-0160, USA\\
email: harrell@math.gatech.edu
\and
Antoine Henrot\\ Institut \'Elie Cartan Nancy\\
UMR 7502,
Nancy Universit\'e - CNRS - INRIA\\
B.P. 70239 54506 Vandoeuvre les Nancy Cedex,  France\\
email: henrot@iecn.u-nancy.fr}

\begin{document}

\maketitle

\begin{abstract}
We consider a family of functionals $J$ to be maximized over the planar convex sets $K$
for which the perimeter and Steiner point have been fixed.  Assuming that $J$ is the integral of a quadratic expression in the support function $h$, we show that the maximizer is always either a triangle or a line segment (which can be considered as a
collapsed triangle).  Among the concrete consequences of the main theorem is the
fact that, given any convex body $K_1$ of finite perimeter, the set in the class we consider that is farthest away in the sense of the $L^2$ distance
is always a line segment. We also prove the same property for the Hausdorff distance.
\end{abstract}

\vspace{0.6cm}

\vspace{0.3cm}


\section{Introduction}
Given a convex set $K_1$ in the plane, consider the problem of finding a second convex set that is as far as possible from $K_1$ in the sense of usual distances like the Hausdorff distance or the $L^2$ distance, subject to two natural geometric constraints, {\it viz.}, that the two sets have the same perimeter and Steiner point, without either of which conditions there are sets arbitrarily far away
from $K_1$.  A plausible conjecture, which we prove below, is that the farthest convex set, subject to the two constraints, is always a ``needle,'' to use the colorful terminology of
P\'olya and Szeg\H{o} \cite{PS} for a line segment in the plane.

In the case of the $L^2$ distance, the problem of the farthest convex set can be expressed as the maximization of a quadratic integral functional of the support function of the desired set, and, as we shall show, with the same two geometric constraints it is
possible to characterize the maximizers of a wider class of such functionals as either triangles or needles, which, intuitively, can be considered as collapsed triangles.  One of our inspirations for pursuing the wider class of functionals, the maximizers of which are triangles, is a recent article \cite{LaNo},
in which the maximizers of another class of convex functionals were shown to be polygons.  Now, the maximizers of a convex functional must lie on the boundary of the feasible set, which is to say, in our case or that of \cite{LaNo}, that the maximizers will be nonstrictly convex, but not simple polygons {\it a priori}.  What restrictions are needed on the functional to imply furthermore that the maximizer must be triangular?  In this article, we consider functionals that are expressible as integrals of quadratic expressions in the support function, and show that the maximizers are always generalized triangles, i.e., triangles or needles.

An advantage of describing shape-optimization problems through the support function
$h$ is that it is easy to express many geometric features, including perimeter and area, in terms of $h$.  Yet another tool that is available to in the case of functionals that are quadratic in $h$ is that of Fourier series \cite{Gro}, because through the Parseval relation it is possible to rewrite many such functionals as series with geometric properties
accessible through the form of the coefficients.  Indeed another one of our inspirations was
the analysis of the maximizers of the $L^2$ means of chord lengths of curves
through Fourier series found in \cite{EHL,EFH}.  When the means with respect to
arc length are replaced with means weighted by curvature, the problem falls
within the category of quadratic functionals of $h$ considered in this article.  Interestingly, the cases of optimality of the weighted and unweighted problems are
completely different.
Because additional analysis is possible for quadratic functionals when the coefficients in
the equivalent series enjoy certain properties, we shall defer details on the chord problem
to a future article
\cite{HH2}.

This paper is organized as follows: We begin Section \ref{sec2} with the main notation and general optimality conditions. We state our main result in Subsection \ref{sec2.3}.
Next, Section \ref{sec3} is devoted to the problem of finding the farthest convex set.
We begin with an inequality involving the minimum and the maximum of the support function,
in the spirit of \cite{McM}. Then, we consider the case of the Hausdorff distance and we
finish with the case of the $L^2$ distance, for which our main result is essential.

\section{Notation and preliminary results}\label{sec2}
\subsection{Notation}
When convenient $\R^2$ will be identified with
the complex plane, and the dot product of two vectors ${\bf x}$ and ${\bf w}$
with $\Re({x \, {\overline w}})$.
Let $\T$ be the unit circle, identified with $[0,2\pi)$.
For $\te\in \T$, we will denote by $h_K(\te)$ (or more simply
$h(\te)$ if not ambiguous)
the support function of the convex set $K$; we recall that by definition
$h(\te)$ is the distance from the
origin to the support line of $K$ having outward unit normal $e^{i\te}$:
$$h_K(\te):=\max \{x \cdot e^{i\te}: x\in K\}\,.$$
It is known that the boundary of a planar convex set has at most a
countable number of points of nondifferentiability.  More precisely,
the two directional derivatives of the function defining any portion
of the boundary exist at every point and their difference is
uniformly bounded.  We refer to \cite{Roc,Web} for this and other
standard facts about convex regions.  It follows from the regularity
of the boundary that the support function $h$ belongs to the
periodic Sobolev space $H^1(\T)$.

For a polygon $K$ with $n$ sides, we let  $a_1,a_2,\ldots,a_n$
and $\te_1,\te_2,\ldots,\te_n$
denote the lengths of the sides and the angles
of the corresponding outer normals. The following characterization of
the support function of such a polygon is classical and will be useful
here:
\begin{proposition}
With the above notation, the support function of the polygon $K$
satisfies
\begin{equation}\label{1.1}
\frac{d^2 h_K}{d\te^2}+h_K=\sum_{j=1}^n a_j \delta_{\te_j}
\end{equation}
where the derivative is to be understood in the sense of distributions
and $\delta_{\te_j}$ stands for a Dirac measure at point $\te_j$.
\end{proposition}

Eq. \eqref{1.1} can be proved by a direct calculation. It is a special case of a
formula of Weingarten, whereby for
any support function $h_K$ of a convex set $K$, $\frac{d^2 h_K}{d\te^2}+h_K
=h_K''+h_K$ is a nonnegative measure, which is interpreted as the
(generalized) radius of curvature $R$
at the point of contact with the support line corresponding to $\te$.
We will denote by $S_h$ (or $S_K$ if we want to emphasize the dependence
on the convex set $K$) the support of this measure. It will be useful to
recover the support function from the radius of curvature.
This can be accomplished by
solving the ordinary differential equation:
\begin{equation}\label{Wein}
h''+h=R
\end{equation}
for a $2 \pi$-periodic function $h(\te)$ subject to the
conditions
\begin{equation}\label{orthcond}
\int_0^{2\pi} h(\te)\cos\te\,d\te=\int_0^{2\pi} h(\te)\sin\te\,d\te=0\,.
\end{equation}
These orthogonality conditions are imposed because
\eqref{Wein}
is in the second Fredholm alternative and hence
needs such conditions for uniqueness.  They can always be
arranged by a choice of the
origin, {\it viz.}, that it is fixed at the Steiner point $s(K)$.  Recall that the
Steiner point
$s(K)$ of a convex planar set $K$ is defined by
\begin{equation}\label{Spdef}
s(K)=\frac{1}{\pi}\int_0^{2\pi} h_K(\te) e^{i\te}\,d\te \,.
\end{equation}
By Fredholm's condition for existence the function or measure $R(\te)$
on the right side of \eqref{Wein} must satisfy the same orthogonality, that is,
$$\int_0^{2\pi} R\cos\te\,d\te=\int_0^{2\pi} R\sin\te\,d\te=0.$$
Since these restrictions on the radius of curvature
are necessary conditions in any case for the closure of the
boundary curve of $K$, they are automatically fulfilled.

An explicit Green function can be found to solve \eqref{Wein}
for $h$ in terms of $R$, i.e.,
$G(t) := \frac{1}{2}\,\left(1-\frac{|t|}{\pi}\right)\sin|t|$, in terms of which

\begin{equation}\label{1.3}
h(\te)=\frac{1}{2}
\int_{-\pi}^{\pi}  G(t) R(\te + t)\,dt\,.
\end{equation}

\smallskip\noindent
The perimeter $P(K)$ of the convex set can be easily calculated
from $h_K$:
\begin{equation}\label{1.4}
P(K)=\int_0^{2\pi} h_K(\te)\,d\te \,.
\end{equation}

In this article, we work within the class of convex sets whose Steiner point
is at the origin and whose perimeter $P(K)$ is fixed, at a value that
can be chosen as
$2\pi$ without loss of generality:
\begin{equation}\label{1.6}
 \mathcal{A}:=\{K\;\mbox{convex set in }\R^2, s(K)=O, P(K)=2\pi\}.
\end{equation}
Given that convexity is equivalent to the nonnegativity of the radius of curvature
$R = h''+h$ (in the sense of measures), the geometric set $\mathcal{A}$
can be described in analytic terms
by requiring $h$
to lie in the function space:
\begin{equation}\label{Hdef}
\begin{array}{l}
  \mathcal{H}:=\{h\in H^1(\T),h\geq 0,\;h''+h\geq 0,\;\\
\quad\quad
\int_0^{2\pi}h d\te=2\pi,\;
\int_0^{2\pi}h\cos\te d\te=\int_0^{2\pi}h\sin\te d\te=0\}.
\end{array}
\end{equation}

The class $\A$ contains in particular
``needles,'' i.e., line segments, which we regard as degenerate
convex bodies in the sense that the perimeter of the segment is taken as
twice its length. We shall let $\Sigma_{\alpha}$ designate the segment
$[-i\frac{\pi}{2} e^{i\alpha}, i\frac{\pi}{2} e^{i\alpha}]$. Its support
function is given by
\begin{equation}\label{1.7}
h_{\alpha}(\te):=\frac{\pi}{2}|\sin(\te-\alpha)|\,,
\end{equation}
which satisfies ${h_{\alpha}}''+h_{\alpha} = \pi(\delta_{\alpha}+
\delta_{\pi+\alpha})$.

\subsection{Optimality conditions}

If the goal is to maximize a functional $J$ defined on the geometric
class $\A$, and $J$ is expressible in terms of the support function
$h$, then we may equivalently consider the problem of determining
\begin{equation}\label{1.8}
\max \{J(h): h\in \mathcal{H}\}.
\end{equation}
We may then analytically determine the conditions for optimality of $J$.

The Steiner point $s$ of a closed convex set always lies within the set, and in the
case of a convex
body (a convex set of nonempty interior),
$s$ is an interior point; see, e.g.,
(1.7.6) in \cite{Sch1}.
It follows that the support function of $K$ can vanish only
if $K$ is a segment.  For any convex body in $\A$, $h_K(\te)>0$
for all $\te$.

We next derive the first and second order
optimality conditions assuming that the optimal set is not a segment,
following \cite{LaNo}.
\begin{theorem}\label{theooptimcond}
If $h_0>0$ is a solution of (\ref{1.8}), where $J:H^1(\T)\to\R$ is $C^2$,
then there exist $\xi_0\in H^1(\T)$, $\xi_0 \le 0$, and $\mu\in\R$ such that
\begin{equation}\label{1.9}
 \xi_0=0\;\mbox{on }S_{h_0},
\end{equation}
and $\forall v\in H^1(\T)$,
\begin{equation}\label{1.10}
\left\langle{J'(h_0),v}\right\rangle = \left\langle{\xi_0+{\xi_0}'',v}\right\rangle+
\mu \int_0^{2\pi} v\,d\te \,.
\end{equation}
Moreover, if $v\in H^1(\T)$ such that
$\exists \lambda\in\R$ satisfies
\begin{equation}\label{1.11}
\begin{array}{l}
 v''+v\geq \lambda ({h_0}''+h_0) \\
v\geq \lambda h_0 \\
\left\langle{\xi_0+{\xi_0}'',v}\right\rangle+\mu \int_0^{2\pi} v\,d\te =0.
\end{array}
\end{equation}
then
\begin{equation}\label{1.12}
\left\langle{J''(h_0),v,v}\right\rangle\leq 0\,.
\end{equation}
\end{theorem}
The proof of the foregoing theorem is classical and can be achieved
using standard first and second order optimality conditions in infinite
dimension space as in \cite{MaZo}; we refer to \cite{LaNo} for technical
details.
\begin{remark}\rm
If the optimal domain $K_0$ is a segment, then the optimality condition
is more complicated to write, because the constraint $h\geq 0$ needs to
be taken into account.  Since it will not be needed here, we do not
write the explicit form.
\end{remark}

\subsection{Integral functionals}\label{sec2.3}
In this section, we are interested in quadratic functionals involving
the support function and its first derivative. Let $J$ be the functional
defined by:
\begin{equation}\label{intfnl}
 J(K):=\int_0^{2\pi}{a\, h_K^2+ b\, {h'_K}^2 + c\, h_K
+ d\, {h'_K}\,d\te},
\end{equation}
where $a$ and $b$ are nonnegative bounded functions of
$\te$, one of them being positive almost everywhere
on $\T$. The functions $c,d$ are assumed to be bounded.
Our main theorem is the following:
\begin{theorem}\label{theo2.1}
Every local maximizer of the functional $J$ defined in \eqref{intfnl},
within the class $\A$ is either a line segment or a triangle.
\end{theorem}
\begin{proof}
Let $K$ be a local maximizer of the functional $J$.
We have to prove that the support $S_K$ of the measure
$h_K''+h_K$ contains no more than three points.
We follow ideas contained in \cite{LRP} and \cite{LaNo}.

Assume, for the purpose of a contradiction, that $S_K$ contains at least
four points $\te_1 < \te_2 < \te_3 <\te_4$ in $(0,2\pi)$. We solve the four
differential equations
\begin{equation}\label{2.2}
 \left\lbrace
\begin{array}{l}
v_i''+v_i=\delta_{\te_i}\quad \te\in (\te_1-\ep,\te_4+\ep)\\
v_i(\te_1-\ep)=v_i(\te_4+\ep)=0,
\end{array}
\right.
\end{equation}
where $\delta_{\te_i}$ is the Dirac measure at point $\te_i$ and $\ep>0$
is chosen such that $\te_4+\ep - (\te_1-\ep)<2\pi$. Note that equations
\eqref{2.2} have unique solutions since we avoid the first eigenvalue of
the interval. We also extend each function $v_i$ by 0 outside
$(\te_1-\ep,\te_4+\ep)$. Now we can always find four numbers $\lambda_i$,
$i=1,\ldots,4$ such that the three following conditions hold, where we denote
by $v$ the function defined by $v=\sum_{i=1}^4\lambda_i v_i$:
\begin{equation}\label{2.3}
 v'(\te_1-\ep)=v'(\te_4+\ep)=0,\quad \int_0^{2\pi} v\,d\te=0\,.
\end{equation}
Then the function $v$ solves $v''+v=\sum_{i=1}^4\lambda_i \delta_{\te_i}$
globally on $(0,2\pi)$. Now, we use the optimality conditions \eqref{1.9},
\eqref{1.10} for the function $v$. We have
$$<\xi_0+{\xi_0}'',v>+
\mu \int_0^{2\pi} v\,d\te=<v''+v,\xi_0>=
\sum_{i=1}^4\lambda_i \xi_0(\te_i)=0\,.$$
Therefore, $v$ is admissible for the second order optimality condition
(it is immediate to check that the two first conditions in \eqref{1.11} are
satisfied by choosing $\lambda<0$ with $|\lambda|$ large enough). Since
the functional $J$ is quadratic, however, this would imply
$\int_0^{2\pi} a{v'}^2+bv^2\,d\te\leq 0$ which
is impossible by the assumptions on $a$ and $b$.
\end{proof}
\begin{remark} \rm
The examples given in the next section may give the impression that the
maximizers for such functionals are always segments. This is not the case.
Indeed, if we choose $a=c=d=0$ and $b$ a (positive) function equal to one
in a $\ep$ neighborhood of $0,2\pi/3$ and $4\pi/3$ and very small elsewhere,
the value for the equilateral triangle is of order $12\pi^2\ep/27$
while the value for the best segment is of order $\pi^2\ep/4$.
\end{remark}

\section{The farthest convex set}\label{sec3}
\subsection{Introduction}
There are many ways to define the distance between convex sets.
Among them we single out the classical Hausdorff distance:
$$d_H(K,L):=\max\{\rho(K,L),\rho(L,K)\},$$
where $\rho$ is defined by
$$\rho(A,B):=\sup_{x\in A}\inf_{y\in B} |x-y|$$
(For a survey of possible metrics we refer to
\cite{Gru}; for a detailed study of the Hausdorff distance
see \cite{HP}).
It is remarkable that the Hausdorff distance can also be defined
using the support functions, as $d_H(K,L)=\|h_K-h_L\|_\infty$.
Moreover the support function allows a
definition of the $L^2$ distance, introduced by McClure and Vitale in
\cite{McC-Vi}, by
$$d_2(K,L):=\left(\int_0^{2\pi} |h_K-h_L|^2\,d\te\right)^{1/2}\,.$$
In \cite{McM}, P. McMullen was able to determine the {\it diameter
in the sense of the Hausdorff distance} of the class $\A$ in any dimension.
More precisely, he proved
that all sets in $\A$ are contained
in the ball of radius $\pi/2$ centered at the origin. In terms of the support
function, this means that, for any convex set $K$ in $\A$, the maximum of
$h_K$ is at most $\pi/2$ (or $P(K)/4$). We will need the following more
precise result:
\begin{theorem}\label{theoineqs}
Let $K$ be any plane convex set with its Steiner point at the origin. Then
\begin{equation}\label{2.5}
\max h_K \leq \dfrac{P(K)}{4}\,\leq \min h_K + \max h_K,
\end{equation}
where both inequalities are sharp and saturated by any
line segment.
\end{theorem}
\begin{proof}
The first inequality in \eqref{2.5} is due to McMullen,
who proved it in any dimension; see
Theorem 1 in \cite{McM}.
Let us prove the second inequality. Letting $B$
denote the unit ball, we introduce
$$\max h_K=\tau(K):=\min\{\tau >0 / K\subset \tau B\}\,,$$
$$\min h_K=\rho(K):=\max\{\rho >0 / \rho B\subset K\}\,.$$
The function $\tau(K)$ is convex with respect to the Minkowski sum,
which can be defined with the support function via
$$
h_{a K + b L} = a h_K + b h_L.
$$
In contrast, the
function $\rho(K)$ is concave, and as we are interested in the sum
$F(K):=\tau(K)+\rho(K)$ we can call upon no particular convexity
property.
The minimum of $h_K$ is attained at some point we call $P$
and the maximum at some point $Q$ (see Figure \ref{figure1}). Let us denote by $L$
the line containing the points $O$ and $P$ and by $\sigma_L$ the reflection
across $L$. If we replace the convex set $K$ by $\frac{1}{2} K +
\frac{1}{2} \sigma_L(K)$, we keep the Steiner point at the origin, we preserve
the perimeter, and we decrease $\tau$, because of convexity, without
changing $\rho$. Therefore, to look for minimum of $F(K)$, we
can restrict ourselves to convex sets symmetric with respect to the line $L$
passing through the point where $h_K$ attains its minimum. Now, let
$S$ be the segment in the class $\A$ which is orthogonal to the line $L$.

We introduce the family of convex sets $K_t:=tK+(1-t)S$ and study the behavior
of $t\mapsto F(K_t)$. Since the ball $t\rho(K) B$ is included in $K_t$ and
touches its boundary at $tP$, we have $\rho(K_t)=t\rho(K)$. Moreover, by
convexity $\tau(K_t)\leq t\tau(K) +(1-t)\tau(S)$. Therefore,
since $\tau(S)=F(S)$
\begin{equation}\label{2.5b}
 F(K_t)\leq tF(K)+(1-t)F(S)\,.
\end{equation}
In particular, this imples that if $F(K)<F(S)$, we would also have $F(K_t)<F(S)$
for $t$ near $0$. Thus, to prove the result it suffices to prove that a
segment is a local minimizer for $J$. Without loss of generality, we consider
the segment $\Sigma_0$ and perturbations respecting the symmetry with
respect to the line $\te=0$.  Let us therefore consider a perturbation of
the segment $\Sigma_0$,
replacing its ``radius of curvature'' $R_0=\pi(\delta_0+\delta_\pi)$ by
$$R_t=R_0+t[\varphi(x) - (\beta \delta_0+(1-\beta)\delta_\pi)]$$
where $\varphi(x)$ is a non negative measure. Since we can work in the class
of symmetric convex sets, we may assume $\varphi$ to be even. Moreover, we
have to assume that $\int_0^{2\pi} R_t = 2\pi$ and $\int_0^{2\pi} R_t\cos(\te)=0$
(the last relation $\int_0^{2\pi} R_t\sin(\te)=0$ is true by symmetry).
This implies that
\begin{equation}\label{2.5c}
\begin{array}{c}\vspace{2mm}
\int_0^{2\pi} \varphi =1, \quad\mbox{or}\quad \int_0^{\pi} \varphi =
\frac{1}{2}\,,\\
\int_0^{2\pi} \varphi \cos\te=2\beta - 1, \quad\mbox{or}\quad
\beta=\frac{1}{2} + \int_0^{\pi} \varphi \cos\te\,.
\end{array}
\end{equation}
Now, the support function $h_t$ of the perturbed convex set can be obtained
thanks to formulae \eqref{1.3}:
$$h_t(\te)=\frac{\pi}{2}|\sin\te|+t\left\{
\int_{-\pi}^{\pi} G(\tau) \varphi(\te +
\tau)\,d\tau
-\beta G(\te)-(1-\beta)G(\te-\pi)\right\},
$$
where $G$ denotes the Green function.
The function $h_t$ will have its maximum near $\pi/2$, so to first order,
\begin{equation}\label{2.5d}
 \max h_t=h_t(\frac{\pi}{2})+o(t)=\frac{\pi}{2}+t\left\{
\int_{-\pi}^{\pi} G(\tau) \varphi(\tau+\frac{\pi}{2})\,
d\tau - \frac{1}{2} \right\}+o(t)\,.
\end{equation}
In the same way, the minimum of $h_t$ will be attained near $0$ or near $\pi$
so to first order
\begin{equation}\label{2.5e}
\begin{array}{l}
  \min h_t=\min(h_t(0),h_t(\pi))+o(t)=\\
\hspace*{2cm} t\min\left\{
\int_{-\pi}^{\pi} G(\tau) \varphi(\tau)\,
d\tau ,  \int_{-\pi}^{\pi} G(\tau) \varphi(\tau+\pi)\,
d\tau \right\}+o(t)\,.
\end{array}
\end{equation}
Therefore, we have to prove that
\begin{equation}\label{2.5f}
\int_{-\pi}^{\pi} G(\tau) \varphi(\tau+\frac{\pi}{2})\,
d\tau + \int_{-\pi}^{\pi} G(\tau) \varphi(\tau)\,
d\tau - \frac{1}{2} >0
\end{equation}
and
\begin{equation}\label{2.5g}
\int_{-\pi}^{\pi} G(\tau) \varphi(\tau+\frac{\pi}{2})\,
d\tau + \int_{-\pi}^{\pi} G(\tau) \varphi(\tau+\pi)\,
d\tau - \frac{1}{2} >0\,.
\end{equation}
Let us prove for example \eqref{2.5f}; the other inequality is similar. Letting
$$A:=\int_{-\pi}^{\pi} G(\tau) \varphi(\tau+\frac{\pi}{2})\,
d\tau + \int_{-\pi}^{\pi} G(\tau) \varphi(\tau)\,
d\tau= \int_{-\pi}^{\pi} (G(\tau)+G(\tau-\frac{\pi}{2})) \varphi(\tau)\,
d\tau$$
and using the fact that $\varphi$ is even,
$$A=\int_{0}^{\pi} [G(\tau)+G(\tau-\frac{\pi}{2})+G(-\tau)+
G(-\tau-\frac{\pi}{2})] \varphi(\tau)\, d\tau\,.$$
Now, it is elementary to check that the function
$\tau\mapsto G_4(\tau):=G(\tau)+G(\tau-\frac{\pi}{2})+G(-\tau)+
G(-\tau-\frac{\pi}{2})$ is always greater or equal to one (see Figure
\ref{figure2})
\begin{figure}[h!]
\begin{center}
\scalebox{.7}[0.55]{\includegraphics{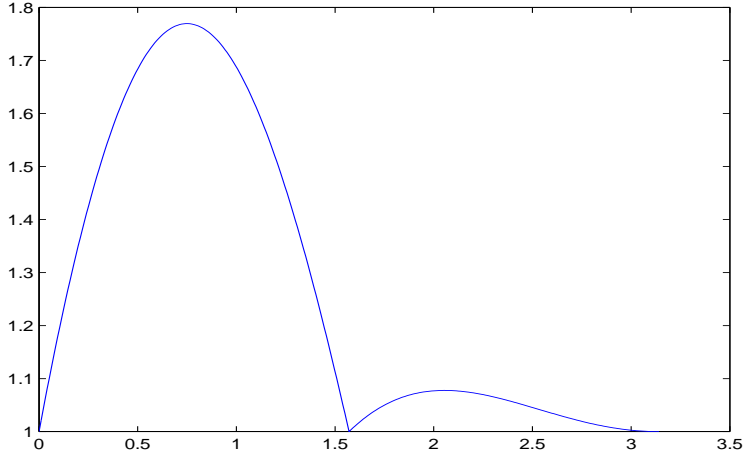}}
\caption{The function $\tau\mapsto G(\tau)+G(\tau-\frac{\pi}{2})+G(-\tau)+
G(-\tau-\frac{\pi}{2}$).\label{figure2}}
\end{center}
\end{figure}
so we have $A\geq \int_{0}^{\pi} \varphi(\tau)\, d\tau=\frac{1}{2}$.
Moreover, since the function $G_4$ is equal to one only for $\tau=0,\pi/2$
or $\pi$, the inequality will be strict unless the support of $\varphi$
is concentrated at the four points $-\pi/2, 0, \pi/2, \pi$. This last case
actually corresponds to a (thin) rectangle
$K_\alpha=[-\alpha,\alpha]\times [-\pi/2+\alpha,\pi/2-\alpha]$ for which a
direct computation shows that $\min h_{K_\alpha}=\alpha/2$ and
$\max h_{K_\alpha}=\left(\alpha^2 + (\pi-\alpha)^2\right)^{1/2}/2$,
and $F(K_\alpha) > \pi/2=F(S)$ follows immediately.
\end{proof}

Another consequence of McMullen's result cited above is
that the Hausdorff distance between two sets in $\A$ is always less or equal
to $\pi/2$, the upper bound being obtained by two orthogonal segments.

\smallskip\noindent
In this section, we want to deal with a similar question, namely to find the
{\it farthest convex set} in the class $\A$ from a given convex set,
as measured by either of the two
distances defined above. More precisely, letting $C$ be a given convex set
in the class $\A$, we wish to find the convex set $K_C$ such that
$$d(C,K_C)=\max\{d(C,K) : K\in \A\},$$
where $d$ may stand either for $d_H$ or for $d_2$.

\smallskip\noindent
First of all, let us give an existence result for such a problem.
\begin{theorem}\label{theoexist1}
Let $d(.,.)$ be a distance function for convex sets that
behaves continuously under
uniform convergence of the support functions. Then the problem
\begin{equation}
 \max\{d(C,K) : K\in \A\}
\end{equation}
has a solution.
\end{theorem}
\begin{proof}
For the proof we will use the following Lemma:
\begin{lemma}\label{lemma1}
For any $h$ in the set ${\mathcal H}$ (defined in (\ref{Hdef})), we have
$$\|h\|^2_{H^1}\leq 16\pi/3\,.$$
\end{lemma}
\begin{proofof}{the Lemma}
For any $h$ in ${\mathcal H}$, we have
\begin{equation}\label{2.10}
 0\leq \int_0^{2\pi} h(h+h'')\,d\te= \int_0^{2\pi} h^2\,d\te -
\int_0^{2\pi} {h'}^2\,d\te\,.
\end{equation}
We now use the fact that the first eigenvalues of the problem
$$
\left\lbrace\begin{array}{c}
 - h''=\lambda h \\
h \mbox{ $2\pi$-periodic}
\end{array}\right.
$$
are $0$ (associated with the constant eigenfunction), $1$ (of multiplicity $2$
associated with
$\sin\te$ and $\cos\te$), $4$ (of multiplicity $2$
associated with
$\sin 2\te$ and $\cos 2\te$). Thus, on $\A$ we can write a minimizing
formula:
\begin{equation}\label{2.11}
 4=\min_{v \in \A}{\left\{\dfrac{\int_0^{2\pi}{v'}^2\,d\te}{\int_0^{2\pi}{v}^2\,d\te}\ s.t.
\int_0^{2\pi}v=\int_0^{2\pi}v\cos\te=\int_0^{2\pi}v\sin\te=0\right\}}\,.
\end{equation}
Applying (\ref{2.11}) to $v=h-1$ yields
$$\int_0^{2\pi}{h'}^2\geq 4 \int_0^{2\pi}(h-1)^2 = 4 \int_0^{2\pi} h^2
- 8\pi\,,$$
or
\begin{equation}\label{2.12}
 \int_0^{2\pi} h^2 \leq \frac{1}{4}\int_0^{2\pi}{h'}^2 + 2\pi\,.
\end{equation}
Combining \eqref{2.10} with \eqref{2.12} leads to
$$\frac{3}{4}\int_0^{2\pi} h^2 \leq 2\pi,$$
and the result follows, once again applying \eqref{2.10} and summing
the two last inequalities.
\end{proofof}

We return to the proof of Theorem \ref{theoexist1}. Let $K_n$ be a
maximizing sequence of convex sets and $h_n$ be the corresponding support
functions. Since the perimeter of $K_n$ is uniformly bounded and the
sets $K_n$ contain the origin, the Blaschke selection theorem applies:
there exists a subsequence, still denoted with the same index, which
converges in the Hausdorff sense to a convex set $K$.
According to Lemma \ref{lemma1}, the support functions $h_n$ are
bounded in $H^1(\T)$, and consequently we
may assume that the sequence converges
uniformly to a function $h$, which is necessarily the support function
of $K$.  Finally, since the distance $d$ has been assumed continuous for this kind of
convergence, the existence of a maximizer follows.
\end{proof}
\medskip
\subsection{The farthest convex set for the Hausdorff distance}
For the Hausdorff distance, we are able to prove that the farthest convex set
is always a segment:
\begin{theorem}
If $C$ is a given convex set
in the class $\A$, then the convex set $K_C$ for which
$$d_H(C,K_C)=\max\{d_H(C,K) : K\in \A\}$$
is a segment. More precisely, it is any segment orthogonal to the
line $OQ$ where $Q$ is any point at which $h_C$ is maximal.
\end{theorem}
\begin{figure}[h!]
\begin{center}
\scalebox{.65}{\includegraphics{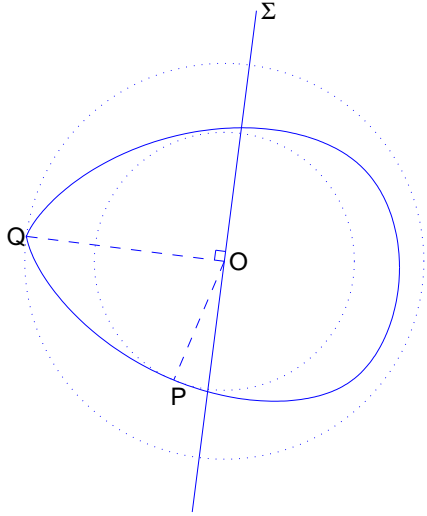}}
\caption{The farthest segment $\Sigma$ for the Hausdorff
distance.\label{figure1}}
\end{center}
\end{figure}

\begin{proof}
 Let $B_1$ be the largest ball centered at $O$ and contained in $C$ and $B_2$
the smallest ball centered at O which contains $C$. We denote by $R_1$
(resp. $R_2$) the radius of $B_1$ (resp. $B_2$). Let $P$, resp. $Q$,
be contact points of these balls with the boundary of $C$ (see Figure
\ref{figure1}). We also denote by $\Sigma_1$ the segment (centered at 0)
containing $P$ and by $\Sigma$ the segment (centered at 0)
orthogonal to $OQ$.

It is easy to see that $\Sigma_1$ is optimal, among all segments $S$, to
maximize $\rho(S,C)$ while $\Sigma$ is optimal to
maximize $\rho(C,S)$. Now, we are going to prove that, for any convex set
$K$ in $\A$:
\begin{equation}\label{2.13}
 \rho(K,C)\leq \rho(\Sigma_1,C)\quad\mbox{and}\quad \rho(C,K)\leq
\rho(C,\Sigma)\,.
\end{equation}
For the first inequality, let us consider any point $M$ in $K$. By
construction of the ball $B_1$:
$$d(M,C)\leq d(M,B_1)=OM-R_1\,.$$
Now, by the first inequality of theorem \ref{theoineqs},
$OM\leq Per(K)/4=\pi/2$ and the result follows taking the supremum in $M$
since $\rho(\Sigma_1,C)=\pi/2-R_1$.

We prove now the second inequality in \eqref{2.13} for any convex body $K$
(the result is already clear for segments as mentioned above). Since
the Steiner point lies in the interior, for any point $M\in\partial C$
$$d(M,K)<OM\leq OQ=\rho(C,\Sigma)\,.$$
Therefore, taking the supremum in $M$, $\rho(C,K)\leq \rho(C,\Sigma)$.\\
From \eqref{2.13} it follows that for any set $K$:
$$d_H(K,C)\leq \max(d_H(\Sigma_1,C),d_H(\Sigma,C))\,.$$
Now, we use the second inequality in Theorem \ref{theoineqs}, which
can be written
$$\rho(\Sigma_1,C)=\pi/2-R_1\leq R_2=\rho(C,\Sigma)\,.$$
Since, however, $\rho(C,\Sigma_1)\leq \rho(C,\Sigma)$, we have
$$d_H(\Sigma_1,C)\leq \rho(C,\Sigma)\leq d_H(\Sigma,C)$$
which gives the desired result.
\end{proof}

\subsection{The farthest convex set for the $L^2$ distance}
For the $L^2$ distance, the result is similar: the
convex set farthest from any given convex set
will be a segment. The proof is more complicated
and relies on our Theorem \ref{theo2.1}.
\begin{theorem}
If $C$ is a given convex set
in the class $\A$, then the convex set $K_C$ for which
$$d_2(C,K_C)=\max\{d_2(C,K) : K\in \A\}$$
is a segment. More precisely, it is any segment $\Sigma_\alpha$
with $\alpha$ which maximizes the one variable function
$\alpha\mapsto \int_0^\pi h_C(\te+\alpha)\sin\te\,d\te$.
\end{theorem}

\begin{proof}
In the proof we denote by $C$ a fixed convex set in the class $\A$.
An immediate consequence of Theorem \ref{theo2.1} applied to
the functional $J$ defined by
$$J(K)=\int_0^{2\pi}(h_K-h_C)^2\,d\te=\int_0^{2\pi}h_K^2- 2h_Ch_K (+h_C^2)
\,d\te$$
is that the farthest convex set is either a triangle or a segment.
Thus, to prove the result, we need to exclude the first possibility.

Let $T$ be a triangle that we assume to be a critical point for the
functional $J:K\mapsto d_2^2(C,K)$. Each triangle in the class $\A$
will be uniquely characterized by its three angles $(\te_1,\te_2,\te_3)$
such that $e^{i\te_k}$ is the normal vector to each side. The only
restrictions we need to put on these angles are
\begin{equation}\label{5.1}
 0<\te_2-\te_1<\pi,\;0<\te_3-\te_2<\pi,\;0<2\pi+\te_1-\te_3<\pi,\,.
\end{equation}
The lengths of the sides will be denoted by $a_1,a_2,a_3$. According
to the law of sines,
given that the perimeter of $T$ is $2\pi$, the three lengths are given by:
\begin{equation}\label{5.2}
\begin{array}{l}
a_1=\dfrac{2\pi sin(\te_3-\te_2)}{sin(\te_3-\te_2)+sin(\te_2-\te_1)+
sin(\te_1-\te_3)},\\
a_2=\dfrac{2\pi sin(\te_1-\te_3)}{sin(\te_3-\te_2)+
sin(\te_2-\te_1)+ sin(\te_1-\te_3)},\\
a_3=\dfrac{2\pi sin(\te_2-
\te_1)}{sin(\te_3-\te_2)+sin(\te_2-\te_1)+sin(\te_1-\te_3)}\,.
\end{array}
\end{equation}
Note that the denominator $sin(\te_3-\te_2)+sin(\te_2-\te_1)+
sin(\te_1-\te_3)$ can also be written $4 sin(\frac{\te_3-\te_2}{2})
sin(\frac{\te_2-\te_1}{2})sin(\frac{\te_1-\te_3}{2})$.

If $A_1,A_2,A_3$ denote the vertices of the triangle, from the relation
$\vec{A_1A_2}+\vec{A_2A_3}+\vec{A_3A_1}=\vec{0}$ rotated by $\pi/2$, we get
\begin{equation}\label{5.3}
 a_1\cos\te_1+a_2\cos\te_2+a_3\cos\te_3=0\quad\mbox{and}\quad
a_1\sin\te_1+a_2\sin\te_2+a_3\sin\te_3=0\,.
\end{equation}
The support function (with the Steiner point at the origin) $h_T(\te)$ of the
triangle $T$ can be calculated with the aid of formula \eqref{1.3} using the
fact that the radius of curvature of $T$ is given by $R=a_1\delta_{\te_1}+
a_2\delta_{\te_2}+a_3\delta_{\te_3}$, according to \eqref{1.1}. One possible
expression for $h$ is:
\begin{equation}\label{5.4}
 h_T(\te)=\left\lbrace
\begin{array}{lc}
\dis\frac{1}{2\pi}\sum_{k=1}^3 a_k\te_k\sin(\te-\te_k),&\te\leq \te_1\,\mbox{or}\,
\te\geq\te_3\\
\dis\frac{1}{2\pi}\sum_{k=1}^3 a_k\te_k\sin(\te-\te_k)+a_1\sin(\te-\te_1),&
\te_1\leq\te\leq \te_2\\
\dis\frac{1}{2\pi}\sum_{k=1}^3 a_k\te_k\sin(\te-\te_k)-a_3\sin(\te-\te_3),&
\te_2\leq\te\leq \te_3,
\end{array}
\right.
\end{equation}
where we have used the fact that, by \eqref{5.3} for any $\te$,
$\sum_{k=1}^3 a_k\sin(\te-\te_k)=0$. We will denote by $\phi(\te)$ the function
$$\phi(\te)=\frac{1}{2\pi}\sum_{k=1}^3 a_k\te_k\sin(\te-\te_k)\,.$$
Now, if $T$ is a critical point of the functional
$\int_0^{2\pi}(h_K-h_C)^2\,d\te$ among any convex set in $\A$, it is also
a critical point among triangles. So we can express that the derivatives
with respect to $\te_1,\te_2,\te_3$ of
$$J(\te_1,\te_2,\te_3)=\int_0^{2\pi}(h_T-h_C)^2\,d\te\,,$$
where $h_T$ is defined in \eqref{5.4}, are zero, that is
$$\int_0^{2\pi}(h_T-h_C)\dfrac{\partial h_T}{\partial \te_j}\,d\te=0,\
j=1,2,3\,.$$
According to \eqref{5.4}, we have (note that $h_T$ is continuous):
\begin{equation}\label{5.5}
\begin{array}{l}\vspace{2mm}
 \dfrac{\partial h_T}{\partial \te_1}=\dfrac{\partial \phi}{\partial \te_1}
+(\dfrac{\partial a_1}{\partial \te_1}\sin(\te-\te_1)-a_1\cos(\te-\te_1))
\chi_{[\te_1,\te_2]}-
\dfrac{\partial a_3}{\partial \te_1}\sin(\te-\te_3)\chi_{[\te_2,\te_3]}\,,\\
\vspace{2mm}
\dfrac{\partial h_T}{\partial \te_2}=\dfrac{\partial \phi}{\partial \te_2}
+\dfrac{\partial a_1}{\partial \te_2}\sin(\te-\te_1)\chi_{[\te_1,\te_2]}-
\dfrac{\partial a_3}{\partial \te_2}\sin(\te-\te_3)\chi_{[\te_2,\te_3]}\,,\\
\dfrac{\partial h_T}{\partial \te_3}=\dfrac{\partial \phi}{\partial \te_3}
+\dfrac{\partial a_1}{\partial \te_3}\sin(\te-\te_1)\chi_{[\te_1,\te_2]}-
(\dfrac{\partial a_3}{\partial \te_3}\sin(\te-\te_3)-a_3\cos(\te-\te_3))
\chi_{[\te_2,\te_3]}\,.
\end{array}
\end{equation}
But since $\frac{\partial \phi}{\partial \te_j}$, for $j=1,2,3$
is a linear combination of
$\sin(\te-\te_k)$ and $\cos(\te-\te_k)$, the contributions
$\int_0^{2\pi}(h_T-h_C)\dfrac{\partial \phi}{\partial \te_k}\,d\te$ are zero
because $\int_0^{2\pi}h\cos\te d\te=\int_0^{2\pi}h\sin\te d\te=0$ for both
$h_T$ and $h_C$.
Therefore, the optimality conditions at the critical triangle $T$ can be written
\begin{equation}\label{5.6}
\left\lbrace\begin{array}{l}
\dfrac{\partial a_1}{\partial \te_1}\int_{\te_1}^{\te_2}(h_T-h_C)\sin(\te-\te_1)
-a_1 \int_{\te_1}^{\te_2}(h_T-h_C)\cos(\te-\te_1) - \\
\hspace*{1cm} \dfrac{\partial a_3}{\partial \te_1}
\int_{\te_2}^{\te_3}(h_T-h_C)\sin(\te-\te_3)
= 0\\
\dfrac{\partial a_1}{\partial \te_2}\int_{\te_1}^{\te_2}(h_T-h_C)\sin(\te-\te_1)
 -
\dfrac{\partial a_3}{\partial \te_2}\int_{\te_2}^{\te_3}(h_T-h_C)\sin(\te-\te_3)
= 0\\
\dfrac{\partial a_1}{\partial \te_3}\int_{\te_1}^{\te_2}(h_T-h_C)\sin(\te-\te_1)
+a_3 \int_{\te_2}^{\te_3}(h_T-h_C)\cos(\te-\te_3) - \\
\hspace*{1cm} \dfrac{\partial a_3}{\partial \te_3}\int_{\te_2}^{\te_3}(h_T-h_C)
\sin(\te-\te_3) = 0\,.\\
\end{array}\right.
\end{equation}
Using \eqref{5.2} we can explicitly compute each partial derivative
$\frac{\partial a_i}{\partial \te_j}$. For example, for $a_1$ they work out to be
\begin{equation}\label{5.7}
\begin{array}{l} \vspace{2mm}
\dfrac{\partial a_1}{\partial \te_2}=\frac{\pi}{2}\cot\frac{\te_1-\te_3}{2}
\frac{1}{\sin^2\frac{\te_2-\te_1}{2}},\quad
\dfrac{\partial a_1}{\partial \te_3}=-\frac{\pi}{2}\cot\frac{\te_2-\te_1}{2}
\frac{1}{\sin^2\frac{\te_1-\te_3}{2}}\\
\dfrac{\partial a_1}{\partial \te_1}=-\dfrac{\partial a_1}{\partial \te_2}-
\dfrac{\partial a_1}{\partial \te_3}=-\frac{\pi}{4}
\dfrac{\sin(\te_1-\te_2)+\sin(\te_1-\te_3)}{\sin^2\frac{\te_2-
\te_1}{2}\sin^2\frac{\te_1-\te_3}{2}}\,.
\end{array}
\end{equation}
In order to simplify the partial derivatives, we introduce the following
integrals:
\begin{equation}\label{5.8}
\begin{array}{ccc}
I_1=\int_{\te_1}^{\te_2}(h_T-h_C)\sin(\te-\te_1) &
I_2=\int_{\te_1}^{\te_2}(h_T-h_C)\sin(\te-\te_2)\\
J_1=\int_{\te_2}^{\te_3}(h_T-h_C)\sin(\te-\te_2) &
J_2=\int_{\te_2}^{\te_3}(h_T-h_C)\sin(\te-\te_3)\\
K_1=\int_{\te_3}^{\te_1+2\pi}(h_T-h_C)\sin(\te-\te_3) &
K_2=\int_{\te_3}^{\te_1+2\pi}(h_T-h_C)\sin(\te-\te_1)
\end{array}
\end{equation}
In consequence, the second equality in \eqref{5.6}
simplifies to:
\begin{equation}\label{5.9}
\dfrac{1}{\sin^2\frac{\te_2-\te_1}{2}}\,I_1+
\dfrac{1}{\sin^2\frac{\te_3-\te_2}{2}}\,J_2=0
\end{equation}
We also introduce the integral
\begin{equation}\label{5.10}
I=\int_{0}^{2\pi}(h_T-h_C)h_T\,d\te
\end{equation}
which is nothing else than half the derivative of the functional $J$
at $h_T$. Using the notation \eqref{5.8} and formulae \eqref{5.4},
together with the fact that $\int_{0}^{2\pi}(h_T-h_C)\phi\,d\te=0$, we get:
$I=a_1I_1-a_3J_2$. Thanks to \eqref{5.2} and \eqref{5.9}, we can express
$I_1$ and $J_2$ in terms of $I$:
\begin{equation}\label{5.11a}
I=-\dfrac{1}{2\sin^2\frac{\te_2-\te_1}{2}}\,I_1=
\dfrac{1}{2\sin^2\frac{\te_3-\te_2}{2}}\,J_2\,.
\end{equation}
Obviously, by symmetry and using other equivalent expressions of the
support function $h_T$, we can also conclude that
\begin{equation}\label{5.11b}
I=-\dfrac{1}{2\sin^2\frac{\te_3-\te_2}{2}}\,J_1=
\dfrac{1}{2\sin^2\frac{\te_1-\te_3}{2}}\,K_2=
-\dfrac{1}{2\sin^2\frac{\te_1-\te_3}{2}}\,K_1=
\dfrac{1}{2\sin^2\frac{\te_2-\te_1}{2}}\,I_2\,.
\end{equation}
Note that we can easily express any of the integrals
$\int_{\te_j}^{\te_{j+1}}(h_T-h_C)\sin\te\,d\te$ or
$\int_{\te_j}^{\te_{j+1}}(h_T-h_C)\cos\te\,d\te$ in terms of the
six integrals defined in \eqref{5.8} and therefore entirely in terms of $I$.

Now summing the three equations in \eqref{5.6} and taking
into account that $\frac{\partial a_1}{\partial \te_1}+
\frac{\partial a_1}{\partial \te_2}+
\frac{\partial a_1}{\partial \te_3}=0$, and the analogous relation for $a_3$, yields
$$a_3\int_{\te_2}^{\te_3}(h_T-h_C)\cos(\te-\te_3)-a_1
\int_{\te_1}^{\te_2}(h_T-h_C)\cos(\te-\te_1)=0\,.$$
We can use the previous expressions to write this last inequality in terms of
the integral $I$, so that
\begin{equation}\label{5.12}
\cos\left(\frac{\te_3-\te_2}{2}\right)(\sin(\te_2-\te_1)-\sin(\te_1-\te_3))\,I
=0\,.
\end{equation}
By symmetry, we get the similar relations obtained by permutation.
Since the cosine is positive (the difference between two angles is less than
$\pi$), we deduce from relation \eqref{5.12} and its analogues that
\begin{enumerate}
 \item either $I=0$
 \item or $\te_3-\te_2=\te_2-\te_1=2\pi+\te_1-\te_3$, that is, $T$ is an
equilateral triangle.
\end{enumerate}
Now, in the case of an equilateral triangle, it is also possible to
simplify the integral $I$. The support function
$h_T$ of the equilateral triangle $\te_1,\te_2=\te_1+2\pi/3,\te_3=\te_1+4\pi/3$
is also given by:
\begin{equation}\label{5.13b}
 h_T(\te)=\left\lbrace
\begin{array}{lc}
\dfrac{2\pi}{3\sqrt{3}}\,\cos(\te-\te_1-\pi/3)& \te_1\leq\te\leq \te_2\\
\dfrac{2\pi}{3\sqrt{3}}\,\cos(\te-\te_1-\pi)& \te_2\leq\te\leq \te_3\\
\dfrac{2\pi}{3\sqrt{3}}\,\cos(\te-\te_1-5\pi/3)& \te_3\leq\te\leq \te_1+2\pi\,.\\
\end{array}
\right.
\end{equation}
Then we have:
$$\begin{array}{l}
I=\dfrac{2\pi}{3\sqrt{3}}\left(\int_{\te_1}^{\te_2}(h_T-h_C)
\cos(\te-\te_1-\pi/3)+\int_{\te_2}^{\te_3}(h_T-h_C)
\cos(\te-\te_1-\pi)\right. \\
\qquad  \ \left. +\int_{\te_3}^{\te_1+2\pi}(h_T-h_C)
\cos(\te-\te_1-5\pi/3)
\right)\,.
  \end{array}$$
Using the notation introduced in \eqref{5.8},
a straightforward computation produces
$$I=\dfrac{2\pi}{9}\left(I_1-I_2 +J_1-J_2+K_1-K_2
\right)\,.
$$
Now, replacing each $I_1,I_2,\ldots$ on the right side by its expression in
terms of $I$ obtained in \eqref{5.11a}, Eq. \eqref{5.11b} yields $I=-2\pi I$.
Thus, we also get $I=0$ in this case.

\medskip\noindent
To conclude the proof, it remains to show that it is impossible
that $I=0$
at a (local) maximum.
Thus, let us assume that $I$, as defined in \eqref{5.10}, is equal to 0. We
consider the family of convex sets $K_t=(1-t)T+t\Sigma_\alpha$ where
$\Sigma_\alpha$ is a segment. The derivative of $t\mapsto J(K_t,C)$
at $t=0$ is $2\int_{0}^{2\pi}(h_T-h_C)(h_{\Sigma_\alpha}-h_T)\,d\te$.
Since $I=0$, this derivative is actually
$$g(\alpha):=\pi\int_{0}^{2\pi}(h_T-h_C)(\te)|\sin(\te-\alpha)|\,d\te\,.$$
We can also write $g(\alpha)$ as
$$g(\alpha):=\pi\int_{0}^{\pi}(h_T-h_C)(\te+\alpha)\sin(\te)\,d\te\,.$$
Now this function of $\alpha$ is $\pi$-periodic, continuous and its
integral over $(0,2\pi)$ is
$$\pi\int_{0}^{2\pi}\int_{0}^{\pi}(h_T-h_C)(\te+\alpha)\sin(\te)\,
d\te d\alpha=0\,.$$
Therefore, either $g(\alpha)$ takes positive and negative values, in which
case
$T$ cannot be a local maximizer, or else $g(\alpha)$ is identically 0.
In the latter case, we come back to the optimality condition (among
all convex sets) given in Theorem \ref{theooptimcond}.
There exist $\xi_0\in H^1(\T)$, nonpositive, vanishing on the support of $T$,
and $\mu\in\R$ such that,
for any $v\in H^1(\T)$, the derivative of the functional is given by
\begin{equation}\label{5.13}
<dJ(T),v>=\int_{0}^{2\pi}(h_T-h_C)v(\te)\,d\te=<\xi_0+{\xi_0}'',v>+
\mu \int_0^{2\pi} v\,d\te \,.
\end{equation}
Applying \eqref{5.13} to $v=h_{\Sigma_\alpha}-h_T$, since the left
side is zero and $\int_{0}^{2\pi} h_{\Sigma_\alpha}=\int_{0}^{2\pi} h_T=2\pi$,
it follows that
for any $\alpha\in (0,\pi)$, $\xi_0(\alpha)+\xi_0(\alpha+\pi)=0$.
Since $\xi_0\leq 0$, this implies that $\xi_0=0$. Now applying
\eqref{5.13} once again to $v=h_{\Sigma_\alpha}$, we get
$$0=\int_{0}^{2\pi}(h_T-h_C)h_{\Sigma_\alpha}\,d\te=2\pi\mu.$$
Thus $\mu=0$ and the derivative of the $L^2$ distance at $T$ is identically
zero. This implies that $C=T$, and is thus actually the global minimizer.

\medskip
The final claim of the theorem follows easily from the expansion
$$\int_0^{2\pi}(h_{\Sigma_\alpha}-h_C)^2\,d\te=\frac{\pi^3}{4}+
\int_0^{2\pi} h_C^2\,d\te-\pi\int_0^{2\pi}h_C|\sin(\te-\alpha)|\,d\te$$
and the equality
$$\int_0^{2\pi}h_C|\sin(\te-\alpha)|\,d\te=2\int_0^{\pi}h_C(\te+\alpha)
\sin\te\,d\te\,.$$
\end{proof}
\begin{remark}\rm
The farthest segment according to the $L^2$ distance is not necessarily unique.
Apart from the trivial example of a disc, for a body of
constant width, {\it every} segment in $\A$ is equally distant. This can easily be seen
using the Fourier series expansion of the support function of
a body of constant width $C$, which is known
to contain only odd terms other than the constant:
$$h_C(\te)=1+\sum_{k=-\infty,\,
k\not= -1,1}^{+\infty} c_{2k+1}e^{(2k+1)i\te},$$
while the Fourier series expansion of the support function $h_{\alpha}$ of a
segment $\Sigma_{\alpha}$ contains only even terms. This is due to the relation
${h''}_{\alpha}+h_{\alpha} = \frac{\pi}{2}(\delta_{\alpha}+
\delta_{\pi+\alpha})$, which when applied to $e^{-in\te}$ yields the following equality
for the $n$-th Fourier coefficient $\gamma_n$ of $h_{\alpha}$:
$$(1-n^2)\gamma_n=\frac{\pi}{2}\,e^{-in\alpha}(1+e^{-in\pi})\,.$$
The $L^2$ distance between $C$ and $\Sigma_{\alpha}$ is
$$d_2(C,\Sigma_{\alpha})=\int_0^{2\pi} h_{\alpha}^2\,d\te-2
\int_0^{2\pi} h_Ch_{\alpha}\,d\te + \int_0^{2\pi} h_C^2\,d\te\,.$$
Now, using the Parseval relation and the orthogonality properties
of the Fourier coefficients
of the two support functions, we see that the integral
$\int_0^{2\pi} h_Ch_{\alpha}\,d\te$ is always equal to $2\pi$, and therefore
the $L^2$ distance between $C$ and a segment does not depend on the segment
within the class $\A$.
\end{remark}
\begin{figure}[h!]
\begin{center}
\scalebox{.7}[0.6]{\includegraphics{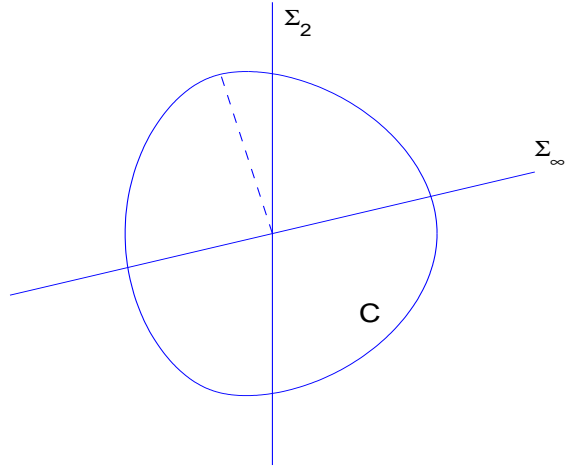}}
\caption{The farthest segments $\Sigma_2$  and $\Sigma_\infty$
do not generally coincide.\label{figure3}}
\end{center}
\end{figure}
\begin{remark}\rm
The farthest segment for the $L^2$ distance and for the Hausdorff distance
do not generally coincide. The Figure \ref{figure3} shows the farthest
segment $\Sigma_2$ (for the $L^2$ distance) and $\Sigma_\infty$ (for the
Hausdorff distance) of the convex set $C$ whose support function is
$h_C(\te)=1-0.1\cos(2\te)+0.05\cos(3\te)$.
\end{remark}

\end{document}